\documentclass[12pt,a4paper,fleqn]{article}
\usepackage{a4wide,amsfonts,amsmath,latexsym,amssymb,euscript,graphicx,units,mathrsfs}

\newtheorem{theorem}{Theorem}[section]

\newtheorem{assumption}[theorem]{Assumption}

\newtheorem{corollary}[theorem]{Corollary}

\newtheorem{definition}[theorem]{Definition}

\newtheorem{lemma}[theorem]{Lemma}

\newtheorem{proposition}[theorem]{Proposition}
\newtheorem{remark}[theorem]{Remark}

\newenvironment{proof}[1][Proof]{\textbf{#1.} }{\ \rule{0.5em}{0.5em}}

\begin{document}

\title{Normal Approximations for Wavelet Coefficients on Spherical Poisson
Fields}
\author{Claudio Durastanti \\
{\small \textit{Department of Mathematics, University of Rome Tor Vergata}}
\and Domenico Marinucci\thanks{%
Corresponding author, email: marinucc@mat.uniroma2.it. Research supported by
the ERC Grant \emph{Pascal,} n.277742} \\
\textit{\small Department of Mathematics, University of Rome Tor Vergata}
\and Giovanni Peccati \\
\textit{{\small Unit\'{e} de Recherche en Math\'{e}matiques, Luxembourg
University } }}
\maketitle

\begin{abstract}
We compute explicit upper bounds on the distance between the law of a
multivariate Gaussian distribution and the joint law of wavelets/needlets
coefficients based on a homogeneous spherical Poisson field. In particular,
we develop some results from Peccati and Zheng (2011), based on Malliavin
calculus and Stein's methods, to assess the rate of convergence to
Gaussianity for a triangular array of needlet coefficients with growing
dimensions. Our results are motivated by astrophysical and cosmological
applications, in particular related to the search for point sources in
Cosmic Rays data.

\medskip


\medskip

\end{abstract}

\section{Introduction}

The aim of this paper is to establish multidimensional normal approximation
results for vectors of random variables having the form of wavelet
coefficients integrated with respect to a Poisson measure on the unit
sphere. The specificity of our analysis is that we require the dimension of
such vectors to grow to infinity. Our techniques are based on recently
obtained bounds for the normal approximation of functionals of general
Poisson measures (see \cite{PSTU, PecZheng}), as well as on the use of the
localization properties of wavelets systems on the sphere (see \cite{npw1},
as well as the recent monograph \cite{mp-book}). A large part of the paper
is devoted to the explicit determination of the above quoted bounds in terms
of dimension.

\subsection{Motivation and overview}

A classical problem in asymptotic statistics is the assessment of the speed
of convergence to Gaussianity (that is, the computation of explicit
Berry-Esseen bounds) for parametric and nonparametric estimation procedures
-- for recent references connected to the main topic of the present paper,
see for instance \cite{HuangZhang, LiWeiXing, YangWangLing}. In this area,
an important novel development is given by the derivation of effective
Berry-Esseen bounds by means of the combination of two probabilistic
techniques, namely the \textit{Malliavin calculus of variations} and the 
\textit{Stein's method} for probabilistic approximations. The monograph \cite%
{ChenGoldShao} is the standard modern reference for Stein's method, whereas 
\cite{np-book} provides an exhaustive discussion of the use of Malliavin
calculus for proving normal approximation results on a Gaussian space. The
fact that one can use Malliavin calculus to deduce normal approximation
bounds (in total variation) for functionals of Gaussian fields was first
exploited in \cite{np-ptrf} -- where one can find several quantitative
versions of the \textquotedblleft fourth moment theorem\textquotedblright\
for chaotic random variables proved in \cite{nualartpeccati2004}. Lower
bounds can also be computed, entailing that the rates of convergence
provided by these techniques are sharp in many instances -- see again \cite%
{np-book}. \newline

In a recent series of contributions, the interaction between Stein's method
and Malliavin calculus has been further exploited for dealing with the
normal approximation of functionals of a general Poisson random measure. The
most general abstract results appear in \cite{PSTU} (for one-dimensional
normal approximations) and \cite{PecZheng} (for normal approximations in
arbitrary dimensions). These findings have recently found a wide range of
applications in the field of stochastic geometry -- see \cite{LRP1, LRP2,
minh, LPST, lesmathias} for a sample of geometric applications, as well as
the webpage 
\begin{equation*}
\mbox{\texttt{http://www.iecn.u-nancy.fr/$\sim$nourdin/steinmalliavin.htm}}
\end{equation*}
for a constantly updated resource on the subject.\newline

The purpose of this paper is to apply and extend the main findings of \cite%
{PSTU, PecZheng} in order to study the multidimensional normal approximation
of the elements of the first Wiener chaos of a given Poisson measure. Our
main goal is to deduce bounds that are well-adapted to deal with
applications where the dimension of a given statistic increases with the
number of observations. This is a framework which arises naturally in many
relevant fields of modern statistical analysis; in particular, our principal
motivation originates from the implementation of \textit{wavelet systems on
the sphere}. In these circumstances, when more and more data become
available, a higher number of wavelet coefficients is evaluated, as it is
customarily the case when considering, for instance, thresholding
nonparametric estimators. We shall hence be concerned with sequences of
Poisson fields, whose intensity grows monotonically. We then exploit the
wavelets localization properties to establish bounds that grow linearly with
the number of functionals considered; we are then able to provide explicit
recipes, for instance, for the number of joint testing procedures that can
be simultaneously entertained ensuring that the Gaussian approximation may
still be shown to hold, in a suitable sense.\newline

\subsection{Main contributions}

Consider a sequence $\{X_{i}:i\geq 1\}$ with values in the unit sphere $%
\mathbb{S}^{2}$, and define $\{\psi _{jk}\}$ to be the collection of the 
\textit{spherical needlets} associated with a certain constant $B>1$, see
Section \ref{ss:needlets} below for more details and discussion. Write also $%
\sigma _{jk}^{2}=E[\psi _{jk}(X_{1})^{2}]$ and $b_{jk}=E[\psi _{jk}(X_{1})]$%
, and consider an independent (possibly inhomogeneous) Poisson process $%
\{N_{t}:t\geq 0\}$ on the real line such that $E[N_{t}]=R(t)\rightarrow
\infty $, as $t\rightarrow \infty $. Formally, our principal aim is to
establish conditions on the sequences $\{j(n):n\geq 1\}$, $\{R(n):n\geq 1\}$
and $\{d(n):n\geq 1\}$ ensuring that the distribution of the centered $d(n)$%
-dimensional vector 
\begin{eqnarray}
&&Y_{n}=(Y_{n,1},...,Y_{n,d(n)})  \label{e:yn} \\
&=&\frac{1}{\sqrt{R(n)}}\left( \sum_{i=1}^{N(n)}\frac{\psi
_{j(n)k_{1}}(X_{i})}{\sigma _{j(n)k_{1}}}-\frac{R(n)b_{j(n)k_{1}}}{\sigma
_{j(n)k_{1}}},...,\sum_{i=1}^{N(n)}\frac{\psi _{j(n)k_{d(n)}}(X_{i})}{\sigma
_{j(n)k_{1}}}-\frac{R(n)b_{j(n)k_{d(n)}}}{\sigma _{j(n)k_{d(n)}}}\right) 
\notag
\end{eqnarray}%
is asymptotically close, in the sense of some smooth distance denoted $d_{2}$
(see Definition \ref{d:2}), to the law of a $d(n)$-dimensional Gaussian
vector, say $Z_{n}$, with centered and independent components having unit
variance. The use of a smooth distance allows one to deduce minimal conditions for this kind of asymptotic Gaussianity. The crucial point is that we allow the dimension $d(n)$ to grow to
infinity, so that our results require to explicitly assess the dependence of
each bound on the dimension. We shall perform our tasks through the
following main steps: (i) Proposition \ref{p:uni} deals with one-dimensional
normal approximations, (ii) Proposition \ref{p:multi} deals with normal
approximations in a fixed dimension, and finally (iii) in Theorem \ref%
{t:growing} we deduce a bound that is well-adapted to the case $%
d(n)\rightarrow \infty $. More precisely, Theorem \ref{t:growing} contains
an upper bound linear in $d(n)$, that is, an estimate of the type 
\begin{equation}
d_{2}(Y_{n},Z_{n})\leq C(n)\times d(n)\text{ .}  \label{e:i}
\end{equation}%
It will be shown in Corollary \ref{c:berlusconiritorna}, that the sequence $%
C(n)$ can be chosen to be 
\begin{equation*}
O\left(1/\sqrt{R(n)B^{-2j(n)}}\right);
\end{equation*}
as discussed below in Remark \ref{importantremark}, $R(n)\times B^{-2j(n)}$
can be viewed as a measure of the ``effective sample size" for the
components of $Y_{n}.$

\subsection{About de-Poissonization}

Our results can be used in order to deduce the asymptotic normality of
de-Poissonized linear statistics with growing dimension. To illustrate this
point, assume that the random variables $X_i$ are uniformly distributed on
the sphere. Then, it is well known that $b_{jk} =0$, whenever $j>1$. In this
framework, when $j(n)>1$ for every $n$, $R(n) = n$ and $d(n)/n^{1/4} \to 0$,
the conditions implying that $Y_n$ is asymptotically close to Gaussian,
automatically ensure that the law of the \textit{de-Poissonized} vector 
\begin{eqnarray}  \label{e:ynprime}
&&Y^{\prime }_n =(Y^{\prime }_{n,1},...,Y^{\prime }_{n,d(n)}) = \frac{1}{%
\sqrt{n}}\left(\sum_{k=1}^{n} \frac{\psi_{j(n)k_1} (X_i)}{{\sigma}_{j(n)k_1}}%
,..., \sum_{k=1}^{n} \frac{\psi_{j(n)k_{d(n)}} (X_i)} {{\sigma}%
_{j(n)k_{d(n)}}}\right)
\end{eqnarray}
is also asymptotically close to Gaussian. The reason for this phenomenon is
nested in the statement of the forthcoming (elementary) Lemma \ref{l:intro}.

\begin{lemma}
\label{l:intro} Assume that $R(n)=n$, that the $X_i$'s are uniformly
distributed on the sphere, and that $j(n)>1$ for every $n$. Then, there
exists a universal constant $M$ such that, for every $n$ and every Lipschitz
function $\varphi : \mathbb{R}^{d(n)} \to \mathbb{R}$, the following
estimate holds: 
\begin{equation*}
\Big|E[\varphi(Y^{\prime }_n)] -E[\varphi(Y_n)]\Big| \leq M\|\varphi\|_{Lip} 
\frac{d(n)}{n^{1/4}}.
\end{equation*}
\end{lemma}

\begin{proof}
Fix $l=1,...,d(n)$, and write $\beta_l(x) = \frac{\psi_{j(n)k_l} (x)}{{\sigma%
}_{j(n)k_l}}$, in such a way that $E[\beta_l(X_1)^2] = 1$. One has that 
\begin{equation*}
E[( Y^{\prime }_{n,l} - Y_{n,l} )^2] = 2 (1-\alpha_n),
\end{equation*}
where 
\begin{equation*}
\alpha_n = \frac{1}{n} \sum_{m=0}^n \frac{e^{-n}n^{m}}{m!}(n\wedge m) = 1- 
\frac{e^{-n}n^{n}}{n!}.
\end{equation*}
This gives the estimate 
\begin{equation*}
E[| Y^{\prime }_{n,l} - Y_{n,l} |] \leq\sqrt{ E[( Y^{\prime }_{n,l} -
Y_{n,l} )^2]} \leq \sqrt{2\frac{e^{-n}n^n}{n!}},
\end{equation*}
so that the conclusion follows from an application of Stirling's formula and
of the Lipschitz property of $\varphi$.
\end{proof}

\begin{remark}
\begin{enumerate}
\item[(i)] \textrm{Lemma \ref{l:intro} implies that one can obtain an
inequality similar to (\ref{e:i}) for $Y^{\prime }_n$, that is: 
\begin{equation*}
d_2(Y^{\prime }_n, Z_n) \leq \left(C(n) +\frac{M}{n^{1/4}}\right) \times
d(n).
\end{equation*}
}

\item[(ii)] \textrm{With some extra work, one can obtain estimates similar
to those in Lemma \ref{l:intro} also when the constants $b_{j(n)k_l}$ are
possibly different from zero. This point, that requires some lenghty
technical considerations, falls slightly outside the scope of this paper and
will be pursued in full generality elsewhere. }

\item[(iii)] \textrm{In \cite{bentkus}, Bentkus proved the following (yet
unsurpassed) bound. Assume that $\{X_{i}:i\geq 1\}$ is a collection of
i.i.d. $d$-dimensional vectors, such that $X_{1}$ is centered and with
covariance equal to the identity matrix. Set $S_{n}=n^{-1/2}(X_{1}+\cdots
X_{n})$, $n\geq 1$ and let $Z$ be a $d$-dimensional centered Gaussian vector
with i.i.d. components having unit variance. Then, for every convex set $%
C\subset \mathbb{R}^{d}$ 
\begin{equation*}
\Big|E[\mathbf{1}_{C}(S_{n})]-E[\mathbf{1}_{C}(Z)]\Big|\leq d^{1/4}\frac{%
400\beta }{\sqrt{n}},
\end{equation*}%
where $\beta =E[\Vert X_{1}\Vert _{\mathbb{R}^{d}}^{3}]$. It is unclear
whether one can effectively use this bound in order to investigate the
asymptotic Gaussianity of sequences of random vectors of the type (\ref{e:yn}%
)--(\ref{e:ynprime}), in particular because, for a fixed $n$, the components
of $Y_{n},\,Y_{n}^{\prime }$ have in general a non trivial correlation. Note
also that a simple application of Jensen inequality shows that $\beta
d^{1/4}n^{-1/2}\geq d^{7/4}n^{-1/2}$. However, a direct comparison of Bentkus' estimates with our
``linear" rate in d (see (\ref{e:i}), as well as Theorem \ref{t:growing}
below) is unfeasible, due to the differences with our setting, namely concerning
the choice of distance, the structure of the considered covariance matrices, the Poissonized environment, and the role of }$B^{j(n)}$ 
\textrm{discussed in Remark \ref{importantremark}}\textrm{\ . }

\item[(iv)] \textrm{A careful inspection of the proofs of our main results
reveals that the findings of this paper have a much more general validity,
and in particular can be extended to kernel estimators on compact spaces
satisfying mild concentration and equispacing properties (see also \cite%
{Jupp1, Jupp2}). In this paper, however, we decided to stick to the
presentation on the sphere for definiteness, and to make the connection with
applications clearer. Some more general frameworks are discussed briefly at
the end of Section \ref{s:dimg1}. }

\item[(v)] For notational simplicity, throughout this paper we will stick to
the case where all the components in our vector statistics are evaluated at
the same scale $j(n)$ (see below for more precise definitions and detailed
discussion). The relaxation of this assumption to cover multiple scales $%
\left( j_{1}(n),...j_{d}(n)\right) $ does not require any new ideas and is
not considered here for brevity's sake.\textrm{\ }
\end{enumerate}
\end{remark}

\subsection{Plan}

The plan of the paper is as follows: in Section \ref{s:prm} we provide some
background material on Stein-Malliavin bounds in the case of Poisson random
fields, and we describe a suitable setting for the current paper, entailing
sequences of fields with monotonically increasing governing measures. We
provide also some new results, ensuring that the Central Limit Theorems we
are going to establish are stable, in the classical sense. In Section \ref%
{s:needletscoeff} we recall some background material on the construction of
tight wavelet systems on the sphere (see \cite{npw1, npw2} for the original
references, as well as \cite[Chapter 10]{mp-book}) and we explain how to
express the corresponding wavelet coefficients in terms of stochastic
integrals with respect to a Poisson random measure. We also illustrate
shortly some possible statistical applications. In Section \ref{s:dim1} we
provide our bounds in the one-dimensional case; these are simple results
which could have been established by many alternative techniques, but still
they provide some interesting insights into the ``effective area of
influence" of a single component of the wavelet system. The core of the
paper is in Section \ref{s:dimg1}, where the bound is provided in the
multidimensional case, allowing in particular for the number of coefficients
to be evaluated to grow with the number of observations. This result
requires a careful evaluation of the upper bound, which is made possible by
the localization properties in real space of the wavelet construction.

\section{Poisson Random Measures and Stein-Malliavin \newline
Bounds}

\label{s:prm}

In order to study the asymptotic behaviour of linear functionals of Poisson
measures on the sphere $\mathbb{S}^{2}$, we start by recalling the
definition of a Poisson random measure -- for more details, see for instance 
\cite{PeTa, privaultbook, SW}. We work on a probability space $(\Omega, 
\mathcal{F}, P)$.

\begin{definition}
\label{d:poisrm} Let $\left( \Theta ,\mathcal{A},\mu \right) $ be a $\sigma $%
-finite measure space, and assume that $\mu$ has no atoms (that is, $%
\mu(\{x\}) = 0$, for every $x\in \Theta$). A collection of random variables $%
\left\{ N\left( A\right) :A\in \mathcal{A}\right\} ,$ taking values in $%
\mathbb{Z}_{+}\cup \left\{ +\infty \right\} $ , is called a \textbf{Poisson
random measure (PRM)} on $\Theta $ with \textbf{intensity measure} (or 
\textbf{control measure}) $\mu $ if the following two properties hold:

\begin{enumerate}
\item For every $A$ $\in \mathcal{A}$, $N\left( A\right) $ has Poisson
distribution with mean $\mu \left( A\right) $;

\item If $A_{1},\ldots A_{n}\in \mathcal{A}$ are pairwise disjoint, then $%
N\left( A_{1}\right) ,\ldots ,N\left( A_{n}\right) $ are independent.
\end{enumerate}
\end{definition}

\begin{remark}
\begin{itemize}
\item[(i)] \textrm{\textrm{In Definition \ref{d:poisrm}, a Poisson random
variable with parameter $\lambda = \infty$ is implicitly set to be equal to $%
\infty$. } }

\item[(ii)] \textrm{\textrm{Points 1 and 2 in Definition \ref{d:poisrm}
imply that, for every $\omega \in \Omega $, the mapping $A \mapsto N\left(
A,\omega \right)$ is a measure on $\Theta $. Moreover, since $\mu $ is non
atomic, one has that 
\begin{equation}  \label{e:supp}
P\big[N(\{x\}) = 0 \mbox{ or } 1,\,\, \forall x \in \Theta \big] =1.
\end{equation}
} }
\end{itemize}
\end{remark}

\begin{assumption}
\label{ass:general}\textrm{Our framework for the rest of the paper will be
the following special case of Definition \ref{d:poisrm}: }

\begin{enumerate}
\item[\textrm{(a)}] \textrm{We take $\Theta = \mathbb{R}_+\times \mathbb{S}%
^2 $, with $\mathcal{A} = \mathcal{B}(\Theta)$, the class of Borel subsets
of $\Theta$. }

\item[\textrm{(b)}] \textrm{The symbol $N$ indicates a Poisson random
measure on $\Theta $, with homogeneous intensity given by $\mu =\rho \times
\nu $, where $\rho $ is some measure on $\mathbb{R}_{+}$ and $\nu $ is a
probability on $\mathbb{S}^{2}$ of the form $\nu (dx)=f(x)dx$, where $f$ is
a density on the sphere. We shall assume that $\rho (\{0\})=0$ and that the
mapping $\rho \mapsto \rho ([0,t])$ is strictly increasing and diverging to
infinity as $t\rightarrow \infty $. We also adopt the notation 
\begin{equation}
R_{t}:=\rho ([0,t]),\quad t\geq 0,  \label{e:erre}
\end{equation}%
that is, $t\mapsto R_{t}$ is the distribution function of $\rho $. }
\end{enumerate}
\end{assumption}

\begin{remark}
\textrm{\label{r:ent} }

\begin{enumerate}
\item[(i)] \textrm{\textrm{For a fixed $t>0$, the mapping 
\begin{equation}
A\mapsto N_{t}(A):=N([0,t]\times A)  \label{e:nt}
\end{equation}%
defines a Poisson random measure on $\mathbb{S}^{2}$, with non-atomic
intensity 
\begin{equation}
\mu _{t}(dx)=R_{t}\cdot \nu (dx)=R_{t}\cdot f(x)dx.  \label{e:mut}
\end{equation}%
Throughout this paper, we shall assume }$f(x)$ to be bounded and bounded
away from zero, e.g.%
\begin{equation}
\zeta _{1}\leq f(x)\leq \zeta _{2}\text{ , some }\zeta _{1},\zeta _{2}>0%
\text{ , for all }x\in \mathbb{S}^{2}\text{ .}  \label{AVA}
\end{equation}
}

\item[(ii)] \textrm{\textrm{Let $\{X_{i}=i\geq 1\}$ be a sequence of i.i.d.
random variables with values in $\mathbb{S}^{2}$ and common distribution
equal to $\nu $. Then, for a fixed $t>0$, the random measure $A\mapsto
N_{t}(A)=N([0,t]\times A)$ has the same distribution as $A\mapsto
\sum_{i=1}^{N}\delta _{X_{i}}(A)$, were $\delta _{x}$ indicates a Dirac mass
at $x$, and $N$ is an independent Poisson random variable with parameter $%
R_{t}$. This holds because: (a) since $\nu $ is a probability measure, the
support of the random measure $N_{t}$ (written $supp(N_{t})$) is almost
surely a finite set, and (b) conditionally on the event $\{N_{t}(\mathbb{S}%
^{2})=n\}$ (which is the same as the event $\{\,|supp(N_{t})|=n\}$ -- recall
(\ref{e:supp})), the points in the support of $N_{t}$ are distributed as $n$
i.i.d. random variables with common distribution $\nu $. } }

\item[(iii)] \textrm{\textrm{By definition, for every $t_1 <t_2$ one has
that a random variable of the type $N_{t_2}(A)- N_{t_1}(A)$, $A\subset 
\mathbb{S}^2$, is independent of the random measure $N_{t_1}$, as defined in
(\ref{e:nt}). } }

\item[(iv)] \textrm{\textrm{To simplify the discussion, one can assume that $%
\rho(ds) = R\cdot \ell(ds)$, where $\ell$ is the Lebesgue measure and $R>0$,
in such a way that $R_t = R\cdot t$. } }
\end{enumerate}
\end{remark}

We will now introduce two distances between laws of random variables taking
values in $\mathbb{R}^{d}$. Both distances define topologies, over the class
of probability distributions on $\mathbb{R}^{d}$, that are strictly stronger
than convergence in law. One should observe that, in this paper, the first
one (Wasserstein distance) will be only used for random elements with values
in $\mathbb{R}$. Given a function $g\in \mathcal{C}^{1}(\mathbb{R}^{d})$, we
write $\Vert g\Vert _{Lip}=\sup\limits_{x\in \mathbb{R}^{d}}\Vert \nabla
g(x)\Vert _{\mathbb{R}^{d}}$. If $g\in \mathcal{C}^{2}(\mathbb{R}^{d})$, we
set 
\begin{equation*}
M_{2}(g)=\sup_{x\in \mathbb{R}^{d}}\Vert \mathrm{Hess}\,g(x)\Vert _{op},
\end{equation*}%
where $\Vert \cdot \Vert _{op}$ indicates the operator norm.

\begin{definition}
\label{d:wass} The \textbf{Wasserstein distance} $d_{W}$, between the laws
of two random vectors $X,Y$ with values in $\mathbb{R}^d$ ($d\geq 1$) and
such that $E\left\| X\right\| _{\mathbb{R}^{d}},E\left\| Y\right\| _{\mathbb{%
R}^{d}}<\infty $, is given by:%
\begin{equation*}
d_{W}\left( X,Y\right) =\sup_{g: \|g\|_{Lip} \leq 1}\left| E\left[ g\left(
X\right) \right] -E\left[ g\left( Y\right) \right] \right| \text{,}
\end{equation*}
\end{definition}

\begin{definition}
\label{d:2} The \textbf{distance} $d_{2}$ between the laws of two random
vectors $X,Y$ with values in $\mathbb{R}^{d}$ ($d\geq 1$), such that $%
E\left\| X\right\|_{\mathbb{R}^{d}},E\left\| Y\right\| _{\mathbb{R}%
^{d}}<\infty $, is given by:%
\begin{equation*}
d_{2}\left( X,Y\right) =\sup_{g\in \mathcal{H}}\left| E\left[ g\left(
X\right) \right] -E\left[ g\left( Y\right) \right] \right| \text{,}
\end{equation*}%
where $\mathcal{H}$ denotes the collection of all functions $g\in \mathcal{C}%
^{2}\left( \mathbb{R}^{d}\right) $ such that $\|g\|_{Lip}\leq 1$ and $%
M_2(g)\leq 1$.
\end{definition}

We now present, in a form adapted to our goals, two upper bounds involving
random variables living in the so-called \textit{first Wiener chaos} of $N$.
The first bound was proved in \cite{PSTU}, and concerns normal
approximations in dimension 1 with respect to the Wasserstein distance. The
second bound appears in \cite{PecZheng}, and provides estimates for
multidimensional normal approximations with respect to the distance $d_2$.
Both bounds are obtained by means of a combination of the Malliavin calculus
of variations and the Stein's method for probabilistic approximations.

\begin{remark}

\begin{itemize}
\item[(i)] \textrm{\textrm{Let $f\in L^{2}(\Theta ,\mu )\cap L^{1}(\Theta
,\mu )$. In what follows, we shall use the symbols $N(f)$ and $\hat{N}(f)$,
respectively, to denote the Wiener-It\^{o} integrals of $f$ with respect to $%
N$ and with respect to the \textit{compensated Poisson measure} 
\begin{equation}
\hat{N}(A)=N(A)-\mu (A),\quad A\in \mathcal{B}(\Theta ),  \label{e:nhat}
\end{equation}%
where one uses the convention $N(A)-\mu (A)=\infty $ whenever $\mu
(A)=\infty $ (recall that $\mu $ is $\sigma $-finite). Note that, for ${N}%
(f) $ to be well-defined, one needs that $f\in L^{1}(\Theta ,\mu )$, whereas
for $\hat{N}(f)$ to be well-defined one needs that $f\in L^{2}(\Theta ,\mu )$%
. We will also make use of the following isometric property: for every $%
f,g\in $\textrm{$L^{2}(\Theta ,\mu )$}, 
\begin{equation}
E[\hat{N}(f)\hat{N}(g)]=\int_{\Theta }f(x)g(x)\mu (dx).  \label{e:isometry}
\end{equation}%
The reader is referred e.g. to \cite[Chapter 5]{PeTa} for an introduction to
Wiener-It\^{o} integrals. } }

\item[(ii)] \textrm{\textrm{For most of this paper, we shall consider
Wiener-It\^{o} integrals of functions $f$ having the form $f=[0,t]\times h$,
where $t>0$ and $h\in L^{2}(\mathbb{S}^{2},\nu )\cap L^{1}(\mathbb{S}%
^{2},\nu )$. For a function $f$ of this type one simply writes 
\begin{equation}
N(f)=N([0,t]\times h):=N_{t}(h),\quad \mbox{and}\quad \hat{N}(f)=\hat{N}%
([0,t]\times h):=\hat{N}_{t}(h).  \label{e:mwiint}
\end{equation}%
Observe that this notation is consistent with the one introduced in (\ref%
{e:nt}). Indeed, it is easily seen that $N_{t}(h)$ (resp. $\hat{N}_{t}(h)$)
coincide with the Wiener-It\^{o} integral of $h$ with respect to $N_{t}$
(resp. with respect to the compensated measure $\hat{N}_{t}=N_{t}-\mu
_{t}=N_{t}-R_{t}\cdot \nu $). } }

\item[(iii)] \textrm{\textrm{In view of Remark \ref{r:ent}-(ii), one also
has that, for $h\in L^{2}(\mathbb{S}^{2},\nu )\cap L^{1}(\mathbb{S}^{2},\nu
) $, 
\begin{equation}
N_{t}(h)=\!\!\!\sum_{x\in supp(N_{t})}h(x),\quad \mbox{and}\,\,\hat{N}%
_{t}(h)=\!\!\!\sum_{x\in supp(N_{t})}h(x)-\int_{\mathbb{S}^{2}}h(x)\mu
_{t}(dx),  \label{e:exmwii}
\end{equation}%
with $\mu _{t}$ defined as in (\ref{e:mut}). } }
\end{itemize}
\end{remark}

\begin{theorem}
\label{t:pstupz} Let the notation and assumptions of this section prevail.

\begin{enumerate}
\item Let $h \in L^2(\mathbb{S}^2, \nu):= L^2(\nu) $, let $Z\sim\mathscr{N}%
(0,1)$ and fix $t>0$. Then, the following bound holds (remember the
definition (\ref{e:mut})): 
\begin{equation}  \label{1stChUB}
d_W(\hat{N}_t(h), Z)\leq \left |1 - \|h\|^2_{L^2(\mathbb{S}^2,\mu_t)} \right
| + \int_{\mathbb{S}^2} |h(z)|^3 \mu_t(dz).
\end{equation}
As a consequence, if $\{h_t\} \subset L^2(\nu) \cap L^3(\nu)$ is a
collection of kernels verifying, as $t\rightarrow \infty$, 
\begin{equation}  \label{1CONDclt}
\|h_t\|_{L^2(\mathbb{S}^2,\mu_t)}\rightarrow 1 \text{\ \ and \ \ }
\|h_t\|_{L^3(\mathbb{S}^2,\mu_t)}\rightarrow 0,
\end{equation}
one has the CLT 
\begin{equation}  \label{1clt}
\hat{N}(h_t) \overset{\mathrm{Law}}{\longrightarrow} Z,
\end{equation}
and the inequality (\ref{1stChUB}) provides an explicit upper bound in the
Wasserstein distance.

\item For a fixed integer $d\geq 1$, let $Y\sim \mathscr{N}_{d}\left(
0,C\right) $, with $C$ positive definite and let 
\begin{equation*}
F_{t}=\left( F_{t,1},\ldots ,F_{t,d}\right) =\left( \hat{N}_{t}\left(
h_{t,1}\right) ,\ldots \hat{N}_{t}\left( h_{t,d}\right) \right)
\end{equation*}%
be a collection of $d$-dimensional random vectors such that $h_{t,a}\in
L^{2}(\nu )$. If we call $\Gamma _{t}$ the covariance matrix of $F_{t}$,
that is, 
\begin{equation*}
\Gamma _{t}\left( a,b\right) =E\left[ \hat{N}_{t}\left( h_{t,a}\right) \hat{N%
}_{t}\left( h_{t,b}\right) \right] =\left\langle
h_{t,a},h_{t,b}\right\rangle _{L^{2}\left( \mathbb{S}^{2},\mu _{t}\right) }%
\text{,}\,\,\,a,b=1,...,d,
\end{equation*}%
then:%
\begin{eqnarray}
d_{2}\left( F_{t},Y\right) &\leq &\left\Vert C^{-1}\right\Vert
_{op}\left\Vert C\right\Vert _{op}^{\frac{1}{2}}\left\Vert C-\Gamma
_{t}\right\Vert _{H.S.}  \label{e:yepes} \\
&&+\frac{\sqrt{2\pi }}{8}\left\Vert C^{-1}\right\Vert _{op}^{\frac{3}{2}%
}\left\Vert C\right\Vert _{op}\sum_{i,j,k=1}^{d}\int_{\mathbb{S}%
^{2}}\left\vert h_{t,i}\left( x\right) \right\vert \left\vert h_{t,j}\left(
x\right) \right\vert \left\vert h_{t,k}\left( x\right) \right\vert \mu
_{t}\left( dx\right) ,  \notag \\
&\leq &\left\Vert C^{-1}\right\Vert _{op}\left\Vert C\right\Vert _{op}^{%
\frac{1}{2}}\left\Vert C-\Gamma _{t}\right\Vert _{H.S.}  \label{d2} \\
&&+\frac{d^{2}\sqrt{2\pi }}{8}\left\Vert C^{-1}\right\Vert _{op}^{\frac{3}{2}%
}\left\Vert C\right\Vert _{op}\sum_{i=1}^{d}\int_{\mathbb{S}^{2}}\left\vert
h_{t,i}\left( x\right) \right\vert ^{3}\mu _{t}\left( dx\right) ,  \notag
\end{eqnarray}%
where $\Vert \cdot \Vert _{op}$ and $\Vert \cdot \Vert _{H.S.}$ stand,
respectively, for the operator and Hilbert-Schmidt norms. In particular, if $%
\Gamma _{t}\left( a,b\right) \longrightarrow C\left( a,b\right) $ and $\int_{%
\mathbb{S}^{2}}\left\vert h_{t,a}\left( x\right) \right\vert ^{3}\mu
_{t}\left( dx\right) \longrightarrow 0$ as $t\longrightarrow \infty $, for $%
a,b=1,\ldots d$, then $d_{2}\left( F_{t},Y\right) \longrightarrow 0$ and $%
F_{t}$ converges in distribution to $Y$.
\end{enumerate}
\end{theorem}

\begin{remark}
\textrm{\textrm{The estimate (\ref{e:yepes}) will be used to deduce one of
the main multidimensional bounds in the present paper. It is a direct
consequence of Theorem 3.3 in \cite{PecZheng}, where the following relation
is proved: for every vector $(F_1,...,F_d)$ of sufficiently regular centered
functionals of $\hat{N}_t$, 
\begin{equation*}
d_{2}(F,X)\leq \left\| C^{-1}\right\| _{op}\left\| C\right\| _{op}^{1/2}%
\sqrt{\sum_{i,j}^{d}\mathbb{E}\left[ C(i,j)-\left\langle
DF_{i},-DL^{-1}F_{j}\right\rangle _{L^{2}(\mu_t )}\right] ^{2}}
\end{equation*}%
\begin{equation*}
+\frac{\sqrt{2\pi }}{8}\left\| C^{-1}\right\| _{op}^{3/2}\left\| C\right\|
_{op}\int_{\mathbb{S}^2}\mu_t (dz)\mathbb{E}\left[ \left(
\sum_{i=1}^{d}\left| D_{z}F_{i}\right| \right) ^{2}\left(
\sum_{j=1}^{d}\left| D_{z}L^{-1}F_{j}\right| \right) \right],
\end{equation*}%
where 
\begin{equation*}
D_{z}F(\omega)=F_{z}(\omega)-F(\omega)\text{ , }a.e.-\mu (dz)P(d\omega) 
\text{ ,}
\end{equation*}%
and 
\begin{equation*}
F_{z}( N )=F_{z}(N +\delta _{z}),
\end{equation*}
that is, the random variable $F_z$ is obtained by adding to the argument of $%
F$ (which is a function of the point measure $N$), a Dirac mass at $z$, and $%
L^{-1}$ is the so-called \textit{pseudo-inverse of the Ornstein-Uhlenbeck
operator}. The estimate (\ref{e:yepes}) is then obtained by observing that,
when $F_i = F_{t,i} = \hat{N}_t(h_{t,i})$, then $D_zF_i = -D_zL^{-1}F =
h_{t,i}(z)$, in such a way that 
\begin{equation*}
\sqrt{\sum_{i,j}^{d}\mathbb{E}\left[ C(i,j)-\left\langle
DF_{i},-DL^{-1}F_{j}\right\rangle _{L^{2}(\mu )}\right] ^{2}} = \left\|
C-K_{t}\right\| _{H.S.},
\end{equation*}
and 
\begin{eqnarray*}
&& \int_{\mathbb{S}^2}\mu_t (dz)\mathbb{E}\left[ \left( \sum_{i=1}^{d}\left|
D_{z}F_{i}\right| \right) ^{2}\left( \sum_{j=1}^{d}\left|
D_{z}L^{-1}F_{j}\right| \right) \right] \\
&& \quad\quad\quad \quad\quad= \sum_{i,j,k=1}^{d}\int_{\mathbb{S}^{2}}\left|
h_{t,i}\left( x\right) \right| \left| h_{t,j}\left( x\right) \right|\left|
h_{t,k}\left( x\right) \right| \mu_t \left( dx\right).
\end{eqnarray*}
}}
\end{remark}

The next statement deals with the interesting fact that the convergence in
law implied by Theorem \ref{t:pstupz} is indeed \textit{stable}, as defined
e.g. in the classic reference \cite[Chapter 4]{JacSh}.

\begin{proposition}
\label{p:stable} The central limit theorem described at the end of Point 2
of Theorem \ref{t:pstupz} (and a fortiori the CLT at Point 1 of the same
theorem) is \textbf{stable} with respect to $\sigma(N)$ (the $\sigma$-field
generated by $N$) in the following sense: for every random variable $X$ that
is $\sigma(N)$-measurable, one has that 
\begin{equation*}
(X,F_t) \overset{\mathrm{Law}}{\longrightarrow} (X, Y),
\end{equation*}
where $Y\sim \mathscr{N}_d(0,C)$ is independent of $N$.
\end{proposition}

\begin{proof}
We just deal with the case $d=1$, the extension to a general $d$ following
from elementary considerations. A density argument shows that it is enough
to prove the following claim: if $\hat{N}(h_{n})$ ($h_{n}\in L^{2}(\mu )$, $%
n\geq 1$) is a sequence of random variables verifying $E[\hat{N}%
(h_{n})^{2}]=\Vert h_{n}\Vert _{L^{2}(\mu )}^{2}\rightarrow 1$ and $%
\int_{\Theta }|h_{n}|^{3}d\mu \rightarrow 0$, then for every fixed $f\in
L^{2}(\mu )$, the pair $(\hat{N}(f),\hat{N}(h_{n}))$ converges in
distribution, as $n\rightarrow \infty $, to $(\hat{N}(f),Z)$, where $Z\sim %
\mathscr{N}(0,1)$ is independent of $N$. To see this, we start with the
explicit formula (see e.g. \cite[formula (5.3.31)]{PeTa}): for every $%
\lambda ,\gamma \in \mathbb{R}$ 
\begin{eqnarray*}
\psi _{n}(\lambda ,\gamma )\!:= &&\!E[\exp (i\lambda \hat{N}(f)+\gamma \hat{N%
}(h_{n}))] \\
&=&\exp \left[ \int_{\Theta }\!\!\left[ e^{i\lambda f(x)+i\gamma
h_{n}(x)}\!-\!1\!-\!i(\lambda f(x)+\gamma h_{n}(x))\right] \mu (dx)\right] .
\end{eqnarray*}%
Our aim is to prove that, under the stated assumptions, 
\begin{equation*}
\lim_{n\rightarrow \infty }\log (\psi _{n}(\lambda ,\gamma ))=\int_{\Theta
}\!\!\left[ e^{i\lambda f(x)}-1-i\lambda f(x)\right] \mu (dx)-\frac{\gamma
^{2}}{2}.
\end{equation*}%
Standard computations show that 
\begin{eqnarray*}
&&\Big|\log (\psi _{n}(\lambda ,\gamma ))-\big\{ \int_{\Theta }\!\!\left[
e^{i\lambda f(x)}-1-i\lambda f(x)\right] \mu (dx)-\frac{\gamma ^{2}}{2}%
\big\} \Big| \\
&\leq &\Big|\frac{\gamma ^{2}}{2}-\frac{\gamma ^{2}}{2}\int_{\Theta
}h_{n}(x)^{2}\mu (dx)\Big|+\left\vert \gamma \lambda \right\vert |\langle
h_{n},f\rangle _{L^{2}(\mu )}|+\frac{|\gamma|^3}{6}\int_{\Theta} |
h_n(x)|^3\mu(dx)\text{ }.
\end{eqnarray*}%
Since $\int_{\Theta }|h_{n}(x)|^{3}\mu(dx) \rightarrow 0$ and the mapping $%
n\mapsto \Vert h_{n}\Vert _{L^{2}(\mu )}^{2}$ is bounded, one has that $%
\langle h_{n},f\rangle _{L^{2}(\mu )}\rightarrow 0$, and the conclusion
follows by using the fact that $\Vert h_{n}\Vert _{L^{2}(\mu
)}^{2}\rightarrow 1$ by assumption.
\end{proof}

\section{Needlet coefficients}

\label{s:needletscoeff}

\subsection{Background: the needlet construction}

\label{ss:needlets}

We now provide an overview of the construction of the set of needlets on the
unit sphere. The reader is referred to \cite[Chapter 10]{mp-book} for an
introduction to this topic. Relevant references on this subject are: the
seminal papers \cite{npw1, npw2}, where needlets have been first defined; 
\cite{gm1, gm2,gelmar,gelpes}, among others, for generalizations to
homogeneous spaces of compact groups and spin fiber bundles; \cite%
{bkmpAoS,bkmpBer,spalan,mayeli} for the analysis of needlets on spherical
Gaussian fields, and \cite{mpbb08,pbm06,dela08,feeney} for some (among many)
applications to cosmological and astrophysical issues; see also \cite%
{mcewen,rosca} for other approaches to spherical wavelets construction.

\medskip

\noindent\textit{(Spherical harmonics)} In Fourier analysis, the set of
spherical harmonics 
\begin{equation*}
\left\{ Y_{lm}:l\geq 0,m=-l,...,l\right\}
\end{equation*}
provides an orthonormal basis for the space of square-integrable functions
on the unit sphere $L^{2}\left( \mathbb{S}^{2}, dx \right) :=L^{2}\left( 
\mathbb{S}^{2}\right) $, where $dx$ stands for the Lebesgue measure on $%
\mathbb{S}^2$ (see for instance \cite{adler, kookim, mp-book, steinweiss}).
Spherical harmonics are defined as the eigenfunctions of the spherical
Laplacian $\Delta _{S^{2}}$ corresponding to eigenvalues $-l\left(
l+1\right) $, e.g. $\Delta _{S^{2}}Y_{lm}=-l(l+1)Y_{lm} $, see again \cite%
{mp-book, steinweiss, VMK} for analytic expressions and more details and
properties. For every $l\geq 0$, we define $\mathcal{K}_l$ as the linear
space given by the restriction to the sphere of the polynomials with degree
at most $l$. Plainly, one has that 
\begin{equation*}
\mathcal{K}_l = \bigoplus_{k=0}^l \mathrm{span} \left\{ Y_{km}:
m=-k,...,k\right\},
\end{equation*}
where the direct sum is in the sense of $L^{2}\left( \mathbb{S}^{2}\right)$.

\medskip

\noindent \textit{(Cubature points)} It is well-known that for every integer $%
l=1,2,... $ there exists a finite set of \textit{cubature points} $\mathcal{Q%
}_{l}\subset \mathbb{S}^{2}$, as well as a collection of \textit{weights} $%
\{\lambda _{\eta }\}$, indexed by the elements of $\mathcal{Q}_{l}$, such
that 
\begin{equation*}
\forall f\in \mathcal{K}_{l},\quad \int_{\mathbb{S}^{2}}f(x)dx=\sum_{\eta
\in \mathcal{Q}_{l}}\lambda _{\eta }f(\eta ).
\end{equation*}%
Now fix $B>1$, and write $[x]$ to indicate the integer part of a given real $%
x$. In what follows, we shall denote by $\mathcal{X}_{j}=\{\xi _{jk}\}$ and $%
\{\lambda _{jk}\}$, respectively, the set $\mathcal{Q}_{[2B^{j+1}]}$ and the
associated class of weights. We also write $K_{j}=card\{\mathcal{X}_{j}\}$.
As proved in \cite{npw1, npw2}, cubature points and weights can be chosen to
satisfy%
\begin{equation}
\lambda _{jk}\approx B^{-2j}\text{ },\text{ }K_{j}\approx B^{2j}\text{ ,}
\label{Njdef}
\end{equation}%
where by $a\approx b$, we mean that there exists $c_{1},c_{2}>0$ such that $%
c_{1}a\leq b\leq c_{2}a$ (see also e.g. \cite{bkmpAoSb, pesenson1,pesenson2} and \cite[Chapter 10]{mp-book}).

\medskip

\noindent \textit{(Spherical needlets)} Fix $B>1$ as before, as well as a
real-valued mapping $b$ on $(0,\infty )$. We assume that $b$ verifies the
following properties: (i) the function $b\left( \cdot \right) $ has compact
support in $\left[ B^{-1},B\right] $ (in such a way that the mapping $%
l\mapsto b\left( \frac{l}{B^{j}}\right) $ has compact support in $l\in \left[
B^{j-1},B^{j+1}\right] $) (ii) for every $\xi \geq 1$, $\sum_{j=0}^{\infty
}b^{2}(\xi B^{-j})=1$ (\textit{partition of unit property}), and (iii) $%
b\left( \cdot \right) \in C^{\infty }\left( 0,\infty \right) $. The
collection of spherical needlets $\{\psi _{jk}\}$, associated with $B$ and $%
b(\cdot )$, are then defined as a weighted convolution of the projection
operator\ $L_{l}(\left\langle x,y\right\rangle )=\sum_{m=-l}^{l}\overline{Y}%
_{lm}\left( x\right) Y_{lm}\left( y\right) $, that is 
\begin{equation}
\psi _{jk}\left( x\right) :=\sqrt{\lambda _{jk}}\sum_{l}b\left( \frac{l}{%
B^{j}}\right) L_{l}(\left\langle x,\xi _{jk}\right\rangle )\text{ ,}
\label{need-def}
\end{equation}

\medskip

\noindent \textit{(Localization)} The properties of $b$ entail the following
quasi-exponential \textit{localization property} (see \cite{npw1} or \cite[%
Section 13.3]{mp-book}): for any $\tau =1,2,...$ there exists $\kappa _{\tau
}>0$ such that for any $x\in \mathbb{S}^{2}$, 
\begin{equation}
\left\vert \psi _{jk}(x)\right\vert \leq \frac{\kappa _{\tau }B^{j}}{\left(
1+B^{j}\arccos \left( \left\langle x,\xi _{jk}\right\rangle \right) \right)
^{\tau }}\text{,}  \label{e:mainlocalization}
\end{equation}%
where $d(x,y):=\arccos \left( \left\langle x,y\right\rangle \right) $ is the
spherical distance. From localization, the following bound can be
established on the $L_{p}\left( \mathbb{S}^{2}\right) $ norms: for all $%
1\leq p\leq +\infty $, there exist two positive constants $q_{p}$ and $%
q_{p}^{\prime }$ such that 
\begin{equation}
q_{p}B^{j\left( 1-\frac{2}{p}\right) }\leq \left\Vert \psi _{jk}\right\Vert
_{L_{p}\left( \mathbb{S}^{2}\right) }\leq q_{p}^{\prime }B^{j\left( 1-\frac{2%
}{p}\right) }.  \label{boundnorm}
\end{equation}

\medskip

\noindent \textit{(Needlets as frames)} Finally, the fact that $b$ is a
partition of unit, allows on to deduce the following \textit{reconstruction
formula} (see again \cite{npw1}): for $f\in L^{2}\left( \mathbb{S}%
^{2}\right) $: 
\begin{equation*}
f(x)=\sum_{j,k}\beta _{jk}\psi _{jk}(x)\text{ ,}
\end{equation*}%
where the convergence of the series is in $L^{2}(\mathbb{S}^{2})$, and 
\begin{equation}
\beta _{jk}:=\left\langle f,\psi _{jk}\right\rangle _{L_{2}\left( \mathbb{S}%
^{2}\right) }=\int_{\mathbb{S}^{2}}f\left( x\right) \psi _{jk}\left(
x\right) dx\text{ ,}  \label{needcoeffic}
\end{equation}%
represents the so-called \textit{needlet coefficient} of index $j,k$.

\subsection{Two motivations: density estimates and point sources}

The principal aim of this paper is to establish multidimensional asymptotic
results for some possibly randomized version of random variables of the type 
\begin{equation}
\widehat{\beta }_{jk}=\widehat{\beta }_{jk}^{(n)}=\frac{1}{n}%
\sum_{i=1}^{n}\psi _{jk}\left( X_{i}\right) \text{,}\quad
j=1,2,...,\,\,k=1,...,K_{j},  \label{e:betahat}
\end{equation}%
where the function $\psi _{jk}$ is defined according to (\ref{need-def}),
and $\{X_{i}:i\geq 1\}$ is some adequate sequence of i.i.d. random
variables. We may also study the asymptotic behaviour, as $t\rightarrow
\infty $, of multi-dimensional object of the type $\left\{ \widehat{\beta }%
_{jk},\text{ }k=1,2,...,K_{j}(t)\right\} $, where $t\mapsto K_{j}(t)$ is a
non-decreasing mapping possibly diverging to infinity, and $j$ may change
with $t$. In other words, as happens in realistic experimental
circumstances, we may decide to focus on a growing number of coefficients as
the number of (expected) events increase. Two strong motivations for this
analysis, both coming from statistical applications, are detailed below.

\medskip

\noindent {(\textbf{Density estimates}}) Consider a density function $f$ on
the sphere $\mathbb{S}^{2}$, that is: $f$ is a mapping from $\mathbb{S}^{2}$
into $\mathbb{R}_{+}$, verifying $\int_{\mathbb{S}^{2}}f(x)dx=1$, where $dx$
indicates the Lebesgue measure on $\mathbb{S}^{2}$. Let $\left\{
X_{i}:i=1,...,n\right\} $ be a collection of i.i.d. observations with values
in $\mathbb{S}^{2}$ with common distribution given by $f(x)dx$. A classical
statistical problem, considered for instance by \cite{bkmpAoSb,
kerkyphampic, kueh}, concerns the estimation of $f$ by \textit{%
wavelets/needlets thresholding techniques}. To this aim, keeping in mind the
notation (\ref{e:betahat}), one uses (\cite{DJKP}, \cite{HKPT}) the
following estimator of $f$: 
\begin{equation*}
\widehat{f}(x)=\sum_{jk}\widehat{\beta }_{jk}^{H}\psi _{jk}\left( x\right) 
\text{, }\quad \widehat{\beta }_{jk}^{H}:=\widehat{\beta }_{jk}\mathbb{I}%
_{\left\{ \left\vert \widehat{\beta }_{jk}\right\vert \geq ct_{n}\right\} }%
\text{,}
\end{equation*}%
where $t_{n}=\sqrt{\log n/n}$ and $c$ is a constant to be determined.
Finite-sample approximations on the distributions of $\widehat{\beta }_{jk}$
can then be instrumental for the exact determination of the thresholding
value $ct_{n}$, see e.g. \cite{DJKP, HKPT}.

\medskip

\noindent \textbf{(Searching for point sources)} The joint distribution of
the coefficients $\{\widehat{\beta }_{jk}\}$ (as defined in (\ref{e:betahat}%
)) is required in statistical procedures devised for the research of
so-called \textit{point sources}, again for instance in an astrophysical
context (see for instance \cite{starck}). The physical issue can be
formalized as follows:

\begin{itemize}
\item[--] Under the null hypothesis, we are observing a background of cosmic
rays governed by a Poisson measure on the sphere $\mathbb{S}^{2}$, with the
form of the measure $N_{t}(\cdot )$ defined in (\ref{e:nt}) for some $t>0$.
In particular, $N_{t}$ is built from a measure $N$ verifying Assumption (\ref%
{ass:general}), and the intensity of $\mu _{t}(dx)=E[N_{t}(dx)]$ is given by
the absolutely continuous measure $R_{t}\cdot f(x)dx$, where $R_{t}>0$ and $%
f $ is a density on the sphere. This situation corresponds, for instance, to
the presence of a diffuse background of cosmological emissions.

\item[--] Under the alternative hypothesis, the background of cosmic rays is
generated by a Poisson point measure of the type: 
\begin{equation*}
N_{t}^{\ast }(A)=N_{t}(A)+\sum_{p=1}^{P}N^{(p)} _{t}\int_{A}\delta _{\xi
_{p}}(x)dx\text{ ,}
\end{equation*}%
where $\left\{ \xi _{1},...,\xi _{P}\right\} \subset \mathbb{S}^{2},$ each
mapping $t\mapsto {N}^{(p)} _{t}$ is an independent Poisson process over $%
\left[ 0,\infty \right) $ with intensity $\lambda_p$, and%
\begin{equation*}
\left\{ \int_{A}\delta _{\xi _{p}}(x)dx=1\right\} \Longleftrightarrow
\left\{ \xi _{p}\in A\right\} \text{ .}
\end{equation*}%
\noindent In this case, one has that $N_{t}^{\ast }$ is a Poisson measure
with atomic intensity 
\begin{equation*}
\mu^* _{t}(A) := E[N_t^*(A)]=R_{t} \int_{A}f(x)dx+\sum_{p=1}^{P}\lambda_pt
\cdot \int_{A}\delta _{\xi _{p}}(x)dx\text{ .}
\end{equation*}
\end{itemize}

In this context, the informal expression \textquotedblleft searching for
point sources\textquotedblright\ can then be translated into
\textquotedblleft testing for $P=0"$ or \textquotedblleft jointly testing
for $\lambda _{p}>0$ at $p=1,...,P".$ The number $P$ and the locations $%
\left\{ \xi _{1},...\xi _{P}\right\} $ can be in general known or unknown.
We refer to \cite{iuppa, scodeller2} for astrophysical applications of these
ideas.

\begin{remark}
\textrm{\textrm{In order to directly apply the findings of \cite{PSTU,
PecZheng}, in what follows we shall focus on a randomized version of (\ref%
{e:betahat}), where $n$ is replaced by an independent Poisson number whose
parameter diverges to infinity. Also, we will prefer a deterministic
normalization over a random one. As formally shown in the discussion to
follow, the resulting randomized coefficients can be neatly put into the
framework of Section \ref{s:prm}. }}
\end{remark}

\subsection{Needlet coefficients as Wiener-It\^o integrals}

Let $N$ be a Poisson measure on $\mathbb{R}_{+}\times \mathbb{S}^{2}$
satisfying the requirements of Assumption \ref{ass:general} (in particular,
the intensity of $N$ has the form $\rho \times \nu $, where $\nu (dx)=f(x)dx$%
, for some probability density $f$ on the sphere, and one writes $R_{t}=\rho
([0,t])$, $t>0$). For every $t>0$, let the Poisson measure $N_{t}$ on $%
\mathbb{S}^{2}$ be defined as in (\ref{e:nt}). For every $j\geq 1$ and every 
$k=1,...,N_{j}$, consider the function $\psi _{jk}$ defined in (\ref%
{need-def}), and observe that $\psi _{jk}$ is trivially an element of $L^{3}(%
\mathbb{S}^{2},\nu )\cap L^{2}(\mathbb{S}^{2},\nu )\cap L^{1}(\mathbb{S}%
^{2},\nu )$. We write 
\begin{equation*}
\sigma _{jk}^{2}:=\int_{\mathbb{S}^{2}}\psi _{jk}^{2}\left( x\right) f(x)dx%
\text{ , }b_{jk}:=\int_{\mathbb{S}^{2}}\psi _{jk}\left( x\right) f(x)dx\text{
}.
\end{equation*}%
Observe that, if $f(x)=\frac{1}{4\pi }$ (that is, the uniform density on the
sphere), then $b_{jk}=0$ for every $j>1$. On the other hand, under (\ref{AVA}%
), 
\begin{equation}
\zeta _{1}\left\Vert \psi _{jk}\left( .\right) \right\Vert _{L^{2}}^{2}\leq
\sigma _{jk}^{2}\leq \zeta _{2}\left\Vert \psi _{jk}\left( .\right)
\right\Vert _{L^{2}}^{2}\text{ .}  \label{ultima}
\end{equation}
Note that (see \ref{boundnorm}) the $L^{2}$-norm of $\left\{ \psi
_{jk}\right\} $ is uniformly bounded above and below, and therefore the same
is true for $\left\{ \sigma _{jk}^{2}\right\} $ (indeed, there exists $%
\kappa >0,$ independent of $j$ and $k,$ such that $0<\kappa <\left\Vert \psi
_{jk}\right\Vert _{L^{2}\left( \mathbb{S}^{2}\right) }^{2}<1).$ For every $%
t>0$ and every $j,k$, we introduce the kernel 
\begin{equation}
h_{jk}^{\left( R_{t}\right) }(x)=\frac{\psi _{jk}\left( x\right) }{\sqrt{%
R_{t}}\sigma _{jk}},\quad x\in \mathbb{S}^{2}\text{ ,}  \label{crucdef1}
\end{equation}%
and write%
\begin{equation}
\widetilde{\beta }_{jk}^{\left( R_{t}\right) }:=\hat{N}_{t}\left(
h_{jk}^{\left( R_{t}\right) }\right) \!\!=\!\!\int_{\mathbb{S}%
^{2}}h_{jk}^{\left( R_{t}\right) }\left( x\right) \hat{N}_{t}\left(
dx\right) \ =\!\!\!\sum_{x\in \mathrm{supp}(N_{t})}h_{jk}^{\left(
R_{t}\right) }(x)-R_{t}\cdot \int_{\mathbb{S}^{2}}h_{jk}^{\left(
R_{t}\right) }(x)\nu (dx)\text{ ,}  \label{crucdef2}
\end{equation}%
In view of Remark \ref{r:ent}-(ii), the random variable $\widetilde{\beta }%
_{jk}^{\left( R_{t}\right) }$ can always be represented in the form 
\begin{equation*}
\widetilde{\beta }_{jk}^{\left( R_{t}\right) }=\frac{\left( \sum_{i=1}^{{N}%
_{t}(\mathbb{S}^{2})}\psi _{jk}\left( X_{i}\right) -R_{t}b_{jk}\right) }{%
\sqrt{R_{t}}\sigma _{jk}},
\end{equation*}%
where $\{X_{i}:i\geq 1\}$ is a sequence of i.i.d. random variables with
common distribution $\nu $, and independent of the Poisson random variable $%
\hat{N}_{t}(\mathbb{S}^{2})$. Moreover, the following relations are
immediately checked: 
\begin{equation}
E\widetilde{\beta }_{jk}^{\left( R_{t}\right) }=0\text{, }\quad E[(%
\widetilde{\beta }_{jk}^{\left( R_{t}\right) })^{2}]=1\text{ .}
\label{varbeta}
\end{equation}

\begin{remark}
\textrm{\textrm{Using the notation (\ref{e:betahat}), we have that%
\begin{equation*}
\widetilde{\beta }_{jk}^{\left( R_{t}\right) }=\frac{\left( N_{t}(\mathbb{S}%
^{2})\times \widehat{\beta }_{jk}^{(N_{t}(\mathbb{S}^{2}))}-R_{t}b_{jk}%
\right) }{\sqrt{R_{t}}\sigma _{jk}}\text{ .}
\end{equation*}%
}}
\end{remark}

\section{Bounds in dimension one}

\label{s:dim1}

We are now going to apply the content of Theorem \ref{t:pstupz}-(1) to the
random variables $\widetilde{\beta }_{jk}^{\left( R_{t}\right) }$ introduced
in the previous section. In the next statement, we write $Z\sim \mathcal{N}%
(0,1)$ to indicate a centered Gaussian random variable with unit variance.
Recall that $\zeta _{2}:=\sup_{x\in \mathbb{S}^{2}}\left\vert f\left(
x\right) \right\vert $, $p\geq 1$, and that the constants $%
q_{p},q_{p}^{\prime }$ have been defined in (\ref{boundnorm}).

\begin{proposition}
\label{p:uni} For every $j,k$ and every $t>0$, one has that 
\begin{equation*}
d_{W}\left( \widetilde{\beta }_{jk}^{\left( R_{t}\right) },Z\right) \leq 
\frac{(q^{\prime }_{3})^3\zeta _{2}B^{j}}{\sqrt{R_{t}}\sigma _{jk}^{3}}\text{
.}
\end{equation*}%
It follows that for any sequence $(j(n),k(n),t(n)),$ $\widetilde{\beta }%
_{j(n)k(n)}^{\left( R_{t(n)}\right) }$ converges in distribution to $Z$, as $%
n\rightarrow \infty $, provided $B^{2j(n)}=o(R_{t(n)}).$ The convergence is $%
\sigma (N)$-stable, in the sense of Proposition \ref{p:stable}.
\end{proposition}

\begin{proof}
Using (\ref{crucdef1})--(\ref{crucdef2}) together with (\ref{d2}) and (\ref%
{AVA}), 
\begin{eqnarray*}
d_{W}\left( \widetilde{\beta }_{jk}^{\left( R_{t}\right) },Z\right) &\leq
&\int_{\mathbb{S}^{2}}\left\vert h_{jk}^{\left( R_{t}\right) }\left(
x\right) \right\vert ^{3}\mu _{t}\left( dx\right) \\
&=&\frac{R_{t}}{\sqrt{R^3_{t}}\sigma _{jk}^{3}}\int_{\mathbb{S}%
^{2}}\left\vert \psi _{jk}\left( x\right) \right\vert ^{3}f\left( x\right)
dx\leq \frac{\zeta _{2}}{\sqrt{R_{t}}\sigma _{jk}^{3}}\left\Vert \psi
_{jk}\right\Vert _{L^{3}\left( \mathbb{S}^{2}\right) }^{3} \\
&\leq &\frac{(q^{\prime }_{3})^3\zeta _{2}B^{j}}{\sqrt{R_{t}}\sigma _{jk}^{3}%
}\text{,}
\end{eqnarray*}%
where in the last inequality we use the property (\ref{boundnorm}) with $p=3$
to have: 
\begin{equation*}
\left\Vert \psi _{jk}\right\Vert _{L^{3}\left( \mathbb{S}^{2}\right)
}^{3}\leq (q^{\prime }_{3})^3B^{3j\left( 1-\frac{2}{3}\right) }=(q^{\prime
}_{3})^3B^{j}\text{.}
\end{equation*}%
The last part of the statement follows from the fact that the topology
induced by the Wasserstein distance (on the class of probability
distributions on the real line) is strictly stronger than the topology of
convergence in law.
\end{proof}

\begin{remark}
\textrm{\textrm{For $f(x)\equiv \left\{ 4\pi \right\} ^{-1}$ we have 
\begin{equation*}
\sigma _{jk}^{2}=\frac{1}{4\pi }\int_{\mathbb{S}^{2}}\psi _{jk}^{2}\left(
x\right) dx=\left\Vert \psi _{jk}\right\Vert _{L^{2}\left( \mathbb{S}%
^{2}\right) }^{2}\text{ ,}
\end{equation*}%
and more generally, under (\ref{AVA}), 
\begin{equation}  \label{e:henry}
d_{W}\left( \widetilde{\beta }_{jk}^{\left( R_{t}\right) },Z\right) \leq 
\frac{B^{j}}{\sqrt{R_{t}}}\frac{(q^{\prime }_3)^3\zeta _{2}}{\zeta
_{1}^{3/2}\left\Vert \psi _{jk}\right\Vert _{L^{2}\left( \mathbb{S}%
^{2}\right) }^{3/2}} :=\gamma(j,k,t)\ .
\end{equation}%
}}
\end{remark}

\begin{remark}
\label{importantremark}\textrm{\textrm{The previous result can be given the
following heuristic interpretation. The factor $B^{-j}$ can be viewed as the
\textquotedblleft effective scale\textquotedblright\ of the wavelet, i.e. it
is the radius of the region centred at $\xi _{jk}$ where the value of the
wavelet function is not negligible. Because needlets are isotropic, the
\textquotedblleft effective area\textquotedblright\ is of order $B^{-2j}.$
For governing measures with density which is bounded and bounded away from
zero, the expected number of observations on a spherical cap of radius $%
B^{-j}$ around $\xi _{jk}$ is hence given by 
\begin{equation*}
E\left[ card\left\{ X_{i}:d(X_{i},\xi _{jk})\leq B^{-j}\right\} \right]
\simeq R_{t}\int_{d(x,\xi _{jk})\leq B^{-j}}f(x)dx\text{ ,}
\end{equation*}%
\begin{equation*}
\zeta _{1}B^{-2j}R_{t}\leq R_{t}\int_{d(x,\xi _{jk})\leq B^{-j}}f(x)dx\leq
\zeta _{2}B^{-2j}R_{t}\text{ ,}
\end{equation*}%
using \cite[equation (8)]{bkmpBer}. Because the Central Limit Theorem can
hold only when the effective number of observations grows to infinity, the
condition $B^{-2j}R_{t}\rightarrow \infty $ is quite expected. In the
thresholding literature, coefficients are usually considered up to the
frequency $J_{R}$ such that $B^{2J_{R}}\simeq R_{t}/\log R_{t},$ see for
instance \textrm{\cite{HKPT}} and \textrm{\cite{bkmpAoSb}}$;$ under these
circumstances, we have 
\begin{equation*}
d_{2}\left( \widetilde{\beta }_{J_{R}k}^{\left( R_{t}\right) },Z\right)
=O\left( \frac{1}{\sqrt{\log R_{t}}}\right) \longrightarrow 0\text{ for }%
R_{t}\longrightarrow +\infty \text{ .}
\end{equation*}%
Therefore $\widetilde{\beta }_{J_{R}k}^{\left( R_{t}\right) }$ does converge
in law to $Z$.}}
\end{remark}

\section{Multidimensional bounds}

\label{s:dimg1}We are now going to apply Part 2 of Theorem (\ref{t:pstupz})
to the computation of multidimensional Berry-Esseen bounds involving vectors
of needlet coefficients of the type (\ref{crucdef2}). After having proved
some technical estimates in Section \ref{ss:techlemma}, we will consider two
bounds. One is proved in Section \ref{ss:fixedd} by means of (\ref{d2}), and
it is well adapted to the case where the number of needlet coefficients, say 
$d$, is fixed. In Section \ref{ss:growingd}, we shall focus on (\ref{e:yepes}%
), and deduce a bound which is adapted to the case where the number $d$ is
possibly growing to infinity.

\subsection{A technical result}

\label{ss:techlemma}

The following estimate, allowing one to bound the covariance between any two
needlet coefficients, will be used throughout this section. We let the
notation and assumptions of the previous section prevail.

\begin{lemma}
\label{l:covtech} For any $j\geq 1$ and $k_{1}\neq k_{2}\leq K_{j}=card\{%
\mathcal{X}_{j}\}$ and every $\tau >0$, there exists a constant $\widetilde{C%
}_{\tau }>0$, solely depending on $\tau ,$ and such that 
\begin{equation*}
\big |\Gamma _{R_{t}}\left( k_{1},k_{2}\right) \big|:=\big|E\widetilde{\beta 
}_{jk_{1}}^{\left( R_{t}\right) }\widetilde{\beta }_{jk_{2}}^{\left(
R_{t}\right) }\big|\leq \frac{\widetilde{C}_{\tau }\zeta _{2}}{\sigma
_{jk_{1}}\sigma _{jk_{2}}\left( 1+B^{j}d\left( \xi _{jk_{1}},\xi
_{jk_{2}}\right) \right) ^{\tau }}\text{ .}
\end{equation*}
\end{lemma}

\begin{proof}
We focus on $\tau > 2;$ note that the inequality for any fixed value of $%
\tau $ immediately implies the result for all $\tau ^{\prime }<\tau .$ For $%
k_{1}\neq k_{2}$ we have:%
\begin{eqnarray*}
\left\vert \Gamma _{R_{t}}\left( k_{1},k_{2}\right) \right\vert
&=&\left\vert \frac{1}{R_{t}\sigma _{jk_{1}}\sigma _{jk_{2}}}\int_{\mathbb{S}%
^{2}}\psi _{jk_{1}}\left( x\right) \psi _{jk_{2}}\left( x\right) \mu
_{t}(dx)\right\vert \\
&=&\frac{R_{t}}{R_{t}\sigma _{jk_{1}}\sigma _{jk_{2}}}\left\vert \int_{%
\mathbb{S}^{2}}\psi _{jk_{1}}\left( x\right) \psi _{jk_{2}}\left( x\right)
f(x)dx\right\vert \\
&\leq &\frac{\zeta _{2}}{\sigma _{jk_{1}}\sigma _{jk_{2}}}\int_{\mathbb{S}%
^{2}}\left\vert \psi _{jk_{1}}\left( x\right) \right\vert \left\vert \psi
_{jk_{2}}\left( x\right) \right\vert dx\text{.}
\end{eqnarray*}%
Now we can use a classical argument (\cite{npw1},\cite{npw2}\cite{bkmpBer})
to show that, for any $\tau >2$, there exists $C_{\tau }>0$ such that:%
\begin{eqnarray*}
&&\left\langle |\psi _{jk_{1}}|,|\psi _{jk_{2}}|\right\rangle _{L^{2}\left( 
\mathbb{S}^{2}\right) }=\int_{\mathbb{S}^{2}}|\psi _{jk_{1}}\left( x\right)
||\psi _{jk_{2}}|\left( x\right) dx \\
&\leq &\kappa _{\tau }B^{2j}\int_{\mathbb{S}^{2}}\frac{1}{\left(
1+B^{j}d\left( x,\xi _{jk_{1}}\right) \right) ^{\tau }}\frac{1}{\left(
1+B^{j}d\left( x,\xi _{jk_{2}}\right) \right) ^{\tau }}dx\text{ .}
\end{eqnarray*}%
In order to evaluate this integral, we can for instance follow (\cite{npw1}%
), by splitting the sphere $\mathbb{S}^{2}$ into two regions:%
\begin{eqnarray*}
S_{1} &=&\left\{ x\in \mathbb{S}^{2}:d\left( x,\xi _{jk_{1}}\right) >d\left(
\xi _{jk_{1}},\xi _{jk_{2}}\right) /2\right\} \\
S_{2} &=&\left\{ x\in \mathbb{S}^{2}:d\left( x,\xi _{jk2}\right) >d\left(
\xi _{jk_{1}},\xi _{jk_{2}}\right) /2\right\} \text{.}
\end{eqnarray*}%
For what concerns the integral on $S_{1}$, we obtain:%
\begin{equation*}
\int_{S_{1}}\frac{1}{\left( 1+B^{j}d\left( x,\xi _{jk_{1}}\right) \right)
^{\tau }}\frac{1}{\left( 1+B^{j}d\left( x,\xi _{jk_{2}}\right) \right)
^{\tau }}dx\leq \frac{2^{\tau }}{\left( 1+B^{j}d\left( \xi _{jk_{1}},\xi
_{jk_{2}}\right) \right) ^{\tau }}\int_{S_{1}}\frac{dx}{\left(
1+B^{j}d\left( x,\xi _{jk_{2}}\right) \right) ^{\tau }}\text{.}
\end{equation*}%
One also has that 
\begin{eqnarray*}
\int_{S_{1}}\frac{dx}{\left( 1+B^{j}d\left( x,\xi _{jk_{2}}\right) \right)
^{\tau }} &\leq &\int_{\mathbb{S}^{2}}\frac{dx}{\left( 1+B^{j}d\left( x,\xi
_{jk_{2}}\right) \right) ^{\tau }}=2\pi \int_{0}^{\pi }\frac{\sin \vartheta 
}{\left( 1+B^{j}\vartheta \right) ^{\tau }}d\vartheta \leq \\
&\leq &\frac{2\pi }{B^{2j}}\int_{0}^{\infty }\frac{y}{\left( 1+y\right)
^{\tau }}dy\leq \frac{2\pi }{B^{2j}}\left[ \int_{0}^{1}ydy+\int_{1}^{\infty
}y^{1-\tau }dy\right] \leq \\
&\leq &\frac{2\pi C}{B^{2j}}\text{.}
\end{eqnarray*}%
Because calculations on the region $S_{2}$ are exactly the same and because $%
\mathbb{S}^{2}\subset S_{1}\cup S_{2}$, we have that, for some constant $%
\widetilde{C}_{\tau }$ depending on $\tau $, 
\begin{equation*}
\left\langle |\psi _{jk_{1}}|,|\psi _{jk_{2}}|\right\rangle _{L^{2}\left( 
\mathbb{S}^{2}\right) }\leq \frac{\widetilde{C}_{\tau }}{\left(
1+B^{j}d\left( \xi _{jk_{1}},\xi _{jk_{2}}\right) \right) ^{\tau }}\text{ },
\end{equation*}%
yielding the desired conclusion.
\end{proof}

\begin{remark}
\textrm{Assuming that $d\left( \xi _{jk_{1}},\xi _{jk_{2}}\right) >\delta $
uniformly for all $j,$ we have immediately 
\begin{equation*}
\big|E\widetilde{\beta }_{jk_{1}}^{\left( R_{t}\right) }\widetilde{\beta }%
_{jk_{2}}^{\left( R_{t}\right) }\big|\leq \kappa _{\tau ,\zeta _{2}}^{\prime
}\times B^{-j\tau }\text{,}
\end{equation*}
where the constant $\kappa _{\tau ,\zeta _{2}}^{\prime}$ only depend on $%
\tau,\zeta_2$.}
\end{remark}

\begin{remark}
\textrm{The previous Lemma provides a tight bound, of some independent
interest, on the high frequency behaviour of covariances among wavelet
coefficients for Poisson random fields. For Gaussian isotropic random
fields, analogous results were provided by \cite{bkmpAoS}, in the case of
standard needlets (bounded support), and by \cite{spalan}--\cite{mayeli}, in
the \textquotedblleft Mexican\textquotedblright\ case where support may be
unbounded in multipole space. It should be noted how asymptotic
uncorrelation holds in much greater generality for Poisson random fields
than for Gaussian field: indeed in the latter case a regular variation
condition had to be imposed on the tail behaviour of the angular power
spectrum, and in the Mexican case this condition had to be strengthened
imposing an upper bound on the decay of the spectrum itself. The reason for
such discrepancy is easily understood: for Poisson random fields, non
overlapping regions are independent, whence (heuristically) localization in
pixel space is sufficient to ensure asymptotic uncorrelation; on the
contrary, in the Gaussian isotropic case different regions of the field are
correlated at any angular distance, and asymptotic uncorrelation for the
coefficients requires a much more delicate cancellation argument.}
\end{remark}

\subsection{Fixed dimension}

\label{ss:fixedd}

Fix $d\geq 2$ and $j\geq 1$, consider a fixed number of sampling points $%
\left\{ \xi _{jk_{1}},...,\xi _{jk_{d}}\right\} $, and define the associated 
$d$-dimensional vector 
\begin{equation*}
\widetilde{\beta }_{j\cdot }^{\left( R_{t}\right) }:=\big(\widetilde{\beta }%
_{jk_{1}}^{\left( R_{t}\right) },...,\widetilde{\beta }_{jk_{d}}^{\left(
R_{t}\right) }\big),
\end{equation*}%
whose covariance matrix will be denoted by $\Gamma _{t}$ (note that, by
construction, $\Gamma _{t}(i,i)=1$ for every $i=1,...,d$). Our aim is to
apply the rough bound (\ref{d2}) in order to estimate the distance between
the law of $\widetilde{\beta }_{j\cdot }^{\left( R_{t}\right) }$ and the law
of a random Gaussian vector $Z\sim \mathcal{N}_{d}(0,I_{d})$, where $C=I_{d}$
stands for the identity $d\times d$ matrix. Using Lemma \ref{l:covtech}, one
has the following basic estimates:

\begin{eqnarray}
\left\Vert C^{-1}\right\Vert _{op} &=&\left\Vert C\right\Vert _{op}^{\frac{1%
}{2}}=1\text{ ,}  \notag \\
\left\Vert C-\Gamma _{t}\right\Vert _{H.S.} &\leq &\sqrt{\sum_{k_{1}\neq
k_{2}=1}^{d}\left\{ E\left[ \widetilde{\beta }_{jk_{1}}^{\left( R_{t}\right)
}\widetilde{\beta }_{jk_{2}}^{\left( R_{t}\right) }\right] \right\} ^{2}} 
\notag \\
&\leq &d\sup_{k_{1}\neq k_{2}=1,...,d}\frac{1}{\sigma _{jk_{1}}\sigma
_{jk_{2}}}\frac{\widetilde{C}_{\tau }\zeta _{2}}{\left( 1+B^{j}d\left( \xi
_{jk_{1}},\xi _{jk_{2}}\right) \right) ^{\tau }}  \notag \\
&\leq &\frac{d}{\zeta _{1}q_{2}^{2}}\times \frac{\widetilde{C}_{\tau }\zeta
_{2}}{\left( 1+B^{j}\inf_{k_{1}\neq k_{2}=1,...,d}d\left( \xi _{jk_{1}},\xi
_{jk_{2}}\right) \right) ^{\tau }}=A(t)\text{ .}  \label{e:at2}
\end{eqnarray}
\ Applying (\ref{d2}) yields therefore that 
\begin{eqnarray*}
d_{2}\left( \widetilde{\beta }_{j}^{\left( R_{t}\right) },Z\right) &\leq
&A(t)+d^{2}\frac{\sqrt{2\pi }}{8}\sum_{k=1}^{d}R_{t}\int_{\mathbb{S}%
^{2}}\left\vert h_{jk}^{\left( R_{t}\right) }\left( x\right) \right\vert
^{3}f\left( x\right) dx \\
&=&A(t)+d^{2}\frac{\sqrt{2\pi }}{8}\frac{\zeta _{2}R_{t}}{\sqrt{R_{t}^{3}}}%
\sum_{k=1}^{d}\int_{\mathbb{S}^{2}}\frac{\left\vert \psi _{jk}\left(
x\right) \right\vert ^{3}}{\sigma _{jk}^{3}}dx \\
&\leq &A(t)+\frac{d^{3}\zeta _{2}}{\sqrt{R_{t}}\zeta _{1}^{3/2}q_{2}^{3}}%
\frac{\sqrt{2\pi }}{8}\left\Vert \psi _{jk}\right\Vert _{L^{3}\left( \mathbb{%
S}^{2}\right) }^{3} \\
&\leq &A(t)+\frac{(q_{3}^{\prime })^{3}d^{3}\zeta _{2}}{\sqrt{R_{t}}\zeta
_{1}^{3/2}q_{2}^{3}}\frac{\sqrt{2\pi }}{8}B^{j},
\end{eqnarray*}%
where we used (\ref{boundnorm}) and (\ref{ultima}) to yield $\sigma
_{jk}^{3}\geq \zeta _{1}^{3/2}q_{2}^{3}$. We write this result as a separate
statement.

\begin{proposition}
\label{p:multi} Under the above notation and assumptions,

\begin{equation*}
d_{2}\left( \widetilde{\beta }_{j}^{\left( R_{t}\right) },Z\right) \leq 
\frac{d\widetilde{C}_{\tau }\zeta _{2}B^{-j\tau }}{\zeta _{1}q_{2}^{2}\left(
1+\inf_{k_{1}\neq k_{2}=1,...,d}d\left( \xi _{jk_{1}},\xi _{jk_{2}}\right)
\right) ^{\tau }}+\frac{(q_{3}^{\prime })^{3}d^{3}\zeta _{2}}{\sqrt{R_{t}}%
\zeta _{1}^{3/2}q_{2}^{3}}\frac{\sqrt{2\pi }}{8}B^{j}\text{ .}
\end{equation*}
\end{proposition}

Because $\tau $ can be chosen arbitrarily large, it is immediately seen that
the leading term in the $d_{2}$ distance is decaying with the same rate as
in the univariate case, e.g. $B^{j}/\sqrt{R_{t}}.$ Assuming however that $%
d=d_{t},$ i.e. the case where the number of coefficients is itself growing
with $t$, the previous bound may become too large to be applicable. We shall
hence try to establish a tighter bound, as detailed in the next section.

\subsection{Growing dimension}

\label{ss:growingd}

In this section we allow for a growing number of coefficients to be
evaluated simultaneously, and investigate the bounds that can be obtained
under these circumstances. More precisely, we are now focussing on 
\begin{equation*}
\widetilde{\beta }_{j(t)\cdot }^{\left( R_{t}\right) }:=\big(\widetilde{%
\beta }_{j(t)k_{1}}^{\left( R_{t}\right) },...,\widetilde{\beta }%
_{j(t)k_{d_{t}}}^{\left( R_{t}\right) }\big),
\end{equation*}%
where $d_{t}\rightarrow \infty ,$ as $t\rightarrow \infty .$ Throughout the
sequel, we shall assume that the points at which these coefficients are
evaluated satisfy the condition: 
\begin{equation}
\inf_{k_{1}\neq k_{2}=1,...,d_{t}}d\left( \xi _{j(t)k_{1}},\xi
_{j(t)k_{2}}\right) \approx \frac{1}{\sqrt{d_{t}}}\text{ .}  \label{diego}
\end{equation}%
Condition (\ref{diego}) is rather minimal; in fact, the cubature points for
a standard needlet/wavelet construction can be taken to form a maximal $%
(d_{t})^{-1/2}$-net (see \cite{bkmpBer, gm1, npw1, pesenson1} for more
details and discussion). The following result is the main achievement of the
paper.

\begin{theorem}
\label{t:growing} Let the previous assumptions and notation prevail. Then
for all $\tau =2,3...,$ there exist positive constants $c$ and $c^{\prime },$
(depending on $\tau ,\zeta _{1},\zeta _{2}$ but not from $t,j(t),d(t)$) such
that we have%
\begin{equation}
d_{2}\left( \widetilde{\beta }_{j(t).}^{\left( R_{t}\right) },Z\right) \leq 
\frac{cd_{t}}{\left( 1+B^{j(t)}\inf_{k_{1}\neq k_{2}=1,...,d_{t}}d\left( \xi
_{j(t)k_{1}},\xi _{j(t)k_{2}}\right) \right) ^{\tau }}+\frac{\sqrt{2\pi }}{8}%
\frac{c^{\prime }d_{t}B^{j(t)}}{\zeta _{1}^{3/2}q_{2}^{3}\sqrt{R_{t}}}\text{ 
}.  \label{mainbound}
\end{equation}
\end{theorem}

\begin{proof}
In view of (\ref{e:yepes}) and (\ref{e:at2}), we just have to prove that the
quantity 
\begin{equation*}
\frac{\sqrt{2\pi }}{8}\frac{R_{t}}{\zeta _{1}^{3/2}q_{2}^{3}\sqrt{R_{t}^{3}}}%
\sum_{k_{1}k_{2}k_{3}}^{d_{t}}\int_{\mathbb{S}^{2}}\left\vert \psi
_{j(t)k_{1}}\left( z\right) \right\vert \left\vert \psi _{j(t)k_{2}}\left(
z\right) \right\vert \left\vert \psi _{j(t)k_{3}}\left( z\right) \right\vert
f(z)dz\text{{}}
\end{equation*}%
is smaller than the second summand on the RHS of (\ref{mainbound}). Now note
that%
\begin{equation*}
\sum_{k_{1}k_{2}k_{3}}^{d_{t}}\int_{\mathbb{S}^{2}}\left\vert \psi
_{j(t)k_{1}}\left( z\right) \right\vert \left\vert \psi _{j(t)k_{2}}\left(
z\right) \right\vert \left\vert \psi _{j(t)k_{3}}\left( z\right) \right\vert
dz\leq \sum_{\lambda }\int_{\mathcal{B}(\xi _{j(t)\lambda
},B^{-j(t)})}\left\{ \sum_{k}^{d_{t}}\left\vert \psi _{j(t)k}\left( z\right)
\right\vert \right\} ^{3}dz\text{ ,}
\end{equation*}%
where, for any $z\in \mathcal{B}(\xi _{j(t)\lambda },B^{-j(t)})$%
\begin{eqnarray*}
\sum_{k}^{d_{t}}\left\vert \psi _{j(t)k}\left( z\right) \right\vert &\leq
&\sum_{k}^{d_{t}}\frac{C_{\tau }B^{j(t)}}{\left\{ 1+B^{j(t)}d(\xi
_{j(t)k},z)\right\} ^{\tau }} \\
&\leq &C_{\tau }B^{j(t)}+\sum_{k:\xi _{j(t)k}\notin \mathcal{B}(\xi
_{j(t)\lambda },B^{-j(t)})}^{d_{t}}\frac{C_{\tau }B^{j(t)}}{\left\{
1+B^{j(t)}\left[ d(\xi _{j(t)k},\xi _{j(t)\lambda })-d(z,\xi _{j(t)\lambda })%
\right] \right\} ^{\tau }} \\
&\leq &C_{\tau }B^{j(t)}+\sum_{k:\xi _{j(t)k}\notin \mathcal{B}(\xi
_{j(t)\lambda },B^{-j(t)})}^{d_{t}}\frac{C_{\tau }B^{j(t)}}{\left\{
B^{j(t)}d(\xi _{j(t)k},\xi _{j(t)\lambda })\right\} ^{\tau }}\text{ .}
\end{eqnarray*}%
Now for $\xi _{j(t)k}\notin \mathcal{B}(\xi _{j(t)\lambda },B^{-j(t)}),$ $%
x\in \mathcal{B}(\xi _{j(t)k},B^{-j(t)}),$ we have by triangle inequality 
\begin{equation*}
d(\xi _{j(t)k},\xi _{j(t)\lambda })+d(\xi _{j(t)k},x)\geq d(\xi
_{j(t)\lambda },x),
\end{equation*}%
and because 
\begin{equation*}
d(\xi _{j(t)k},\xi _{j(t)\lambda })\geq d(\xi _{j(t)k},x),\text{ and }2d(\xi
_{j(t)k},\xi _{j(t)\lambda })\geq d(\xi _{j(t)\lambda },x)\text{ ,}
\end{equation*}%
we obtain 
\begin{eqnarray*}
&&\sum_{k:\xi _{j(t)k}\notin \mathcal{B}(\xi _{j(t)\lambda
},B^{-j(t)})}^{d_{t}}\frac{C_{\tau }B^{j(t)}}{\left\{ B^{j(t)}d(\xi
_{j(t)k},\xi _{j(t)\lambda })\right\} ^{\tau }} \\
&=&\sum_{k:\xi _{j(t)k}\notin \mathcal{B}(\xi _{j(t)\lambda
},B^{-j(t)})}^{d_{t}}\frac{1}{meas(\mathcal{B}(\xi _{j(t)k},B^{-j(t)}))}%
\int_{\mathcal{B}(\xi _{j(t)k},B^{-j(t)})}\frac{\kappa _{\tau }B^{j(t)}}{%
\left\{ B^{j(t)}d(\xi _{j(t)k},\xi _{j(t)\lambda })\right\} ^{\tau }}dx \\
&\leq &\sum_{k:\xi _{j(t)k}\notin \mathcal{B}(\xi _{j(t)\lambda
},B^{-j(t)})}^{d_{t}}\frac{1}{meas(\mathcal{B}(\xi _{j(t)k},B^{-j(t)}))}%
\int_{\mathcal{B}(\xi _{j(t)k},B^{-j(t)})}\frac{\kappa _{\tau }2^{\tau
}B^{j(t)}}{\left\{ B^{j(t)}d(\xi _{j(t)\lambda },x)\right\} ^{\tau }}dx\leq
\kappa _{\tau }^{\prime }B^{j(t)}\text{ ,}
\end{eqnarray*}%
arguing as in \cite{bkmpAoS}, Lemma 6. Hence%
\begin{equation}
\sum_{k}^{d_{t}}\left\vert \psi _{j(t)k}\left( z\right) \right\vert \leq
\kappa _{\tau }^{\prime \prime }B^{j(t)}\text{ ,}  \label{uniformbound}
\end{equation}%
uniformly over $z\in S^{2},$ which immediately provides the bound. 
\begin{equation*}
\sum_{\lambda }\int_{\mathcal{B}(\xi _{j(t)\lambda },B^{-j})}\left\{
\sum_{k}^{d_{t}}\left\vert \psi _{j(t)k}\left( z\right) \right\vert \right\}
^{3}dz\leq (\kappa _{\tau }^{\prime \prime }B^{j})^{3}\sum_{\lambda }\int_{%
\mathcal{B}(\xi _{j(t)\lambda },B^{-j(t)})}dz=(\kappa ^{\prime \prime \prime
}B^{j(t)})^{3}\text{ .}
\end{equation*}%
Finally, to establish the sharper constraint%
\begin{equation*}
\int_{\mathbb{S}^{2}}\left\{ \sum_{k}^{d_{t}}\left\vert \psi _{j(t)k}\left(
z\right) \right\vert \right\} ^{3}dz\leq \widetilde{\kappa }_{\tau
}d_{t}B^{j(t)},
\end{equation*}%
it is sufficient to note that, exploiting (\ref{uniformbound}) 
\begin{eqnarray*}
&&\sum_{k_{1}}\int_{\mathbb{S}^{2}}\left\vert \psi _{j(t)k_{1}}\left(
z\right) \right\vert \sum_{k_{2}}\left\vert \psi _{j(t)k_{2}}\left( z\right)
\right\vert \sum_{k_{3}}\left\vert \psi _{j(t)k_{3}}\left( z\right)
\right\vert dz \\
&\leq &\kappa ^{2}B^{2j(t)}\sum_{k_{1}}^{d_{t}}\int_{\mathbb{S}%
^{2}}\left\vert \psi _{j(t)k_{1}}\left( z\right) \right\vert dz=\kappa
^{2}B^{2j(t)}d_{j(t)}\left\Vert \psi _{j(t)k}\right\Vert _{L^{1}(\mathbb{S}%
^{2})} \\
&\leq &d_{t}\kappa ^{2}B^{2j(t)}B^{-j(t)}=d_{t}\kappa ^{2}B^{j(t)}\text{ ,}
\end{eqnarray*}%
where we have used again $\left\Vert \psi _{j(t)k}\right\Vert _{L^{p}(%
\mathbb{S}^{2})}^{p}=O(B^{2j(t)(\frac{1}{2}-\frac{1}{p}%
)p})=O(B^{j(t)(p-2)}), $ for $p=1$. Thus (\ref{mainbound}) is established.
\end{proof}

For definiteness, we shall also impose tighter conditions on the rate of
growth of $d_{t},B^{j(t)}$ with respect to $R_{t},$ so that we can obtain a
much more explicit bound, as follows:

\begin{corollary}
\label{c:berlusconiritorna} Let the previous assumptions and notation
prevail, and assume moreover that there exists $\alpha ,\beta $ such that,
as $t\rightarrow \infty $%
\begin{equation*}
B^{2j(t)}\approx R_{t}^{\alpha }\text{ , }0<\alpha <1\text{ , }d_{t}\approx
R_{t}^{\beta }\text{ , }0<\beta <\alpha \text{ .}
\end{equation*}%
There exists a constant $\kappa $ (depending on $\zeta _{1},\zeta _{2}$, but
not on $j,d_{j},B$) such that%
\begin{equation}
d_{2}\left( \widetilde{\beta }_{j(t).}^{\left( R_{t}\right) },Z\right) \leq
\kappa \frac{d_{t}B^{j(t)}}{\sqrt{R_{t}}}\text{ },  \label{corollary}
\end{equation}%
for all vectors $\big(\widetilde{\beta }_{jk_{1}}^{\left( R_{t}\right) },...,%
\widetilde{\beta }_{jk_{d_{t}}}^{\left( R_{t}\right) }\big),$ \ such that (%
\ref{diego}) holds.
\end{corollary}

\begin{proof}
It suffices to note that%
\begin{eqnarray*}
\frac{d_{t}\kappa _{\tau ,\zeta _{2}}^{\prime }}{\left(
1+B^{j(t)}\inf_{k_{1}\neq k_{2}=1,...,d_{t}}d\left( \xi _{jk_{1}},\xi
_{jk_{2}}\right) \right) ^{\tau }} &=&O(B^{-\tau j(t)}d_{t}^{1+\tau /2}) \\
&=&O\left(\frac{d_{t}B^{j(t)}}{\sqrt{R_{t}}}\left(\frac{R_{t}d_{t}^{\tau }}{%
B^{(\tau +1)2j(t)}}\right)^{1/2}\right)
\end{eqnarray*}%
and 
\begin{equation*}
\frac{R_{t}d_{t}^{\tau }}{B^{(\tau +1)2j(t)}}=\frac{R_{t}^{1+\beta \tau }}{%
R_{t}^{(\tau +1)\alpha }}=R_{t}^{-\alpha +\tau (\beta -\alpha )+1}=o(1)\text{
, for }\tau >\frac{1-\alpha }{\alpha -\beta }\text{ .}
\end{equation*}
\end{proof}

\begin{remark}
\textrm{From (\ref{corollary}), it follows that for $R_{t}\simeq 10^{12}$ we
can establish asymptotic joint Gaussianity for all sequences of coefficients 
$\big(\widetilde{\beta }_{j(t)k_{1}(t)}^{\left( R_{t}\right) },...,%
\widetilde{\beta }_{j(t)k_{d}(t)}^{\left( R_{t}\right) })$ of dimensions
such that%
\begin{equation*}
\frac{d_{t}B^{j(t)}}{\sqrt{R_{t}}}=o(1)\text{ ,}
\end{equation*}%
e.g. we can take $d_{t}\simeq o(\sqrt{R_{t}}/B^{j(t)})\simeq
o(10^{6}/B^{j(t)})$, so that even at multipoles in the order of }$%
B^{j(t)}=O(10^{3})$\textrm{\ we might take around 10}$^{\mathrm{3}}$\textrm{%
\ coefficients with the multivariate Gaussian approximation still holding.
These arrays would not be sufficient for the map reconstruction at this
scale, but would indeed provide a basis for joint multiple testing
procedures as those described earlier.}
\end{remark}

\begin{remark}
\textrm{Assume that $d_{t}$ scales as $B^{2j(t)};$ loosely speaking, this
corresponds to the situation when one focusses on the whole set of
coefficients corresponding to scale $j,$ so that exact reconstruction for
bandlimited functions with $l=O(B^{j})$ is feasible. Under this requirement,
however, the "covariance" term $A(t),$ i.e. the first element on the
right-hand side of (\ref{mainbound}), is no longer asymptotically negligible
and the approximation with Gaussian independent variables cannot be expected
to hold. The approximation may however be implemented in terms of a Gaussian
vector with dependent components. For the second term, convergence to zero
when $d_{j(t)}\approx B^{2j(t)}$\ requires $B^{3j(t)}=o(\sqrt{R_{t}})$. In
terms of astrophysical applications, for $R_{t}\simeq 10^{12}$ this implies
that one can focus on scales until $180%
{{}^\circ}%
/B^{j}\simeq 180%
{{}^\circ}%
/10^{2}\simeq 2%
{{}^\circ}%
;$ this is close to the resolution level considered for ground-based Cosmic
Rays experiments such as ARGO-YBJ (\cite{iuppa}). Of course, this value is
much lower than the factor $B^{j}=o(\sqrt{R_{t}})=o(10^{6})$ required for
the Gaussian approximation to hold in the one-dimensional case (e.g., on a
univariate sequence of coefficients, for instance corresponding to a single
location on the sphere).}
\end{remark}

\begin{remark}
\textrm{As mentioned in the introduction, in this paper we decided to focus
on a specific framework (spherical Poisson fields), which we believe of
interest from the theoretical and the applied point of view. It is readily
verified, however, how our results continue to hold with trivial
modifications in a much greater span of circumstances, indeed in some cases
with simpler proofs. Assume for instance we observe a sample of $i.i.d.$
random variables $\left\{ X_{t}\right\} ,$ with probability density function 
$f(.)$ which is bounded and has support in $[a,b]\subset \mathbb{R}$.
Consider the kernel estimates 
\begin{equation}
\widehat{f}_{n}(x_{nk}):=\frac{1}{nB^{-j}}\sum_{t=1}^{n}K(\frac{X_{t}-x_{nk}%
}{B^{-j}})\text{ ,}  \label{stat}
\end{equation}%
where $K(.)$ denotes a compactly supported and bounded kernel satisfying
standard regularity conditions, and for each $j$ the evaluation points $%
(x_{n0},...,x_{nB^{j}})$ form a $B^{-j}$-net; for instance%
\begin{equation*}
a=x_{n0}<x_{n1}...<x_{nB^{j}}=b\text{ , }x_{nk}=a+k\frac{b-a}{B^{j}},\text{ }%
k=0,1,...,B^{j}\text{ .}
\end{equation*}%
As argued earlier, conditionally on $N_{t}([a,b])=n,$ (\ref{stat}) has the
same distribution as 
\begin{equation*}
\widehat{f}_{N_{t}}(x_{nk}):=\frac{1}{N_{t}[a,b]B^{-j}}\int_{a}^{b}K(\frac{%
u-x_{nk}}{B^{-j}})dN_{t}(u)\text{ ,}
\end{equation*}%
where $N_{t}$ is a Poisson measure governed by $R_{t}\times \int_{A}f(x)dx$
for all $A\subset \lbrack a,b].$ Considering that $\frac{N_{t}}{R_{t}}%
\rightarrow _{a.s.}1,$ a bound analogous to \ref{corollary} can be
established with little efforts for the vector $\widehat{f}%
_{n}(x_{n.}):=\left\{ \widehat{f}_{n}(x_{n1}),...,\widehat{f}%
_{n}(x_{nB^{j}})\right\} .$ We leave this and related developments for
further research.}
\end{remark}

\end{document}